\documentclass[reqno]{amsart}
\usepackage{amssymb,latexsym}
\usepackage{amsmath}
\usepackage{amsthm}
\usepackage{graphicx}
\usepackage{hyperref}
\usepackage{titletoc}
\numberwithin{equation}{section}
\newtheorem{theorem}{Theorem}[section]
\newtheorem{prop}[theorem]{Proposition}
\newtheorem{lemma}[theorem]{Lemma}
\newtheorem{corollary}[theorem]{Corollary}

\newtheorem{definition}[theorem]{Definition}

\theoremstyle{definition}

\newcommand{\Z}{\mathbb Z}
\newcommand{\E}{\mathbb E}
\newcommand{\R}{\mathbb R}
\newcommand{\N}{\mathbb N}

\newcommand{\outBd}{\partial_o}
\newcommand{\diam}{\textrm{diam}}

\newtheorem{remark}{Remark}[section]

\newcommand{\ssubset}{\subset\subset}
\newcommand{\Capa}{\text{Cap}}
\newcommand{\BCap}{\text{BCap}}
\newcommand{\EsC}{\text{Es}}

\newcommand{\KRW}{\mathbf k}
\newcommand{\SRW}{\mathbf s}
\newcommand{\BRW}{\mathbf b}
\newcommand{\Snake}{\mathcal S}
\newcommand{\bb}{\Gamma}

\newcommand{\pS}{\mathbf p}
\newcommand{\rS}{\mathbf r}
\newcommand{\qS}{\mathbf q}

\newcommand{\DeFine}{\doteq}
\newcommand{\dist}{\rho}

\newcommand{\Ball}{\mathcal C}
\newcommand{\BallE}{\mathcal B}

\newcommand{\Rad}{\text{Rad}}
\newcommand{\Hm}{\mathcal{H}}

\begin{document}
\title[On CBRW II]{On the critical branching random walk II: Branching capacity and branching recurrence}
\author{Qingsan Zhu}
\address{~Department of Mathematics, University of British
Columbia,
Vancouver, BC V6T 1Z2, Canada}
\email{qszhu@math.ubc.ca}
\date{}
\maketitle

\begin{abstract}
We continue our study of critical branching random walk and branching capacity. In this paper we introduce branching recurrence and branching transience and prove an analogous version of Wiener's Test.
\end{abstract}

\section{Introduction}
In the previous paper \cite{Z161}, we introduce branching capacity $\BCap(K)$ for every finite subset $K$ of $\Z^d$ ($d\geq 5$) and construct its relations with the visiting probability by a critical branching random walk:
\begin{equation}\label{lim-BCap}
\lim_{x\rightarrow \infty}\|x\|^{d-2}\cdot P(\Snake_x \text{ visits } K)= a_d\BCap(K),
\end{equation}
where $\Snake_x$ is the critical branching random walk starting from $x\in \Z^d$, $a_d=\frac{1}{2d^{(d-2)/2}\sqrt{\det Q}}\Gamma(\frac{d-2}{2})\pi^{-d/2}$, $\|\cdot\|$ is a norm corresponding to the jump distribution.

In the present paper, we establish a relation between branching capacity and branching recurrence, i.e. a new version of Wiener's Test. Let us first review the classical Wiener's Test. A subset (finite or infinite) $K\subset\Z^d$ is called recurrent if
$$
P(S_0(n)\in K \text{ infinitely often })=1,
$$
and transient if
$$
P(S_0(n)\in K \text{ infinitely often })=0,
$$
where $(S_0(n))_{n\in\N}$ is the random walk starting from $0$.
For the recurrence and transience of a set, Wiener's Test says that:\\
Suppose $K\subset \Z^d,d\geq3$ and let $K_n=\{a\in K: 2^n\leq |a|< 2^{n+1}\}$. Then,
$$
K \text{ is recurrent } \Leftrightarrow \sum_{n=1}^{\infty}\frac{\Capa(K_n)}{2^{n(d-2)}}=\infty.
$$

Inspired by this, we give the definition of branching recurrence and branching transience by using the incipient infinite branching random walk (also called the branching random walk conditioned on survival, see Section 2 for  the exact definition) instead of random walk. Then we construct the following version of Wiener's Test:
\begin{theorem}\label{MT}
Assume that the offspring distribution $\mu$ on $\N$ is nondegenerate, critical and with finite third moment, and that the jump distribution $\theta$ on $\Z^d$ ($d\geq 5$) is centered, with finite range and not supported on a strict subgroup of $\Z^d$. Then for any $K\subset \Z^d$, we have
$$
K \text{ is branching recurrent } \Leftrightarrow \sum_{n=1}^{\infty}\frac{\BCap(K_n)}{2^{n(d-4)}}=\infty.
$$
\end{theorem}

In the previous paper \cite{Z161}, we give the following bounds for the branching capacity of low dimensional balls:
\begin{equation}
\BCap(B^m(r))\asymp \left\{
\begin{array}{ccc}
r^{d-4}&\text{ if }& m\geq d-3;\\
r^{d-4}/\log r &\text{ if }& m=d-4;\\
r^{m} &\text{ if }& m\leq d-5;\\
\end{array}
\right.
\end{equation}
where $B^m(r)=\{z=(z_1,0)\in\Z^m\times\Z^{d-m}=\Z^d: |z_1|\leq r\}$.

Therefore, if we let $M$ be the $(d-i)$-dimensional ($i=1,2,3,4$) linear subspace, i.e. $\{z=(z_1,z_2)\in\Z^{d-i}\times \Z^i:z_2=0\}$ in $Z^d,d\geq5)$, then by the result above and our main theorem, Theorem \ref{MT} and the monotonicity of branching capacity (see Proposition 4.1 in \cite{Z161}), one can see that $M$ is branching recurrent. By projecting to $\Z^i$, we get that for $\Z^i$ ($i\leq 4$), the projection of the branching random walk conditioned on survival will visit any vertex infinitely often, almost surely. Hence we can get the following result which generalizes the main result in \cite{BC12}:
\begin{corollary}
The incipient infinite branching random walk in $\Z^d$ ($d\leq 4$) almost surely visits any vertex infinitely often, provided that the offspring distribution has finite third moment and that the step distribution is centered, with finite range and not supported on a strict subgroup of $\Z^d$.
\end{corollary}

In the previous paper \cite{Z161} (Theorem 1.5), we show that branching capacities for different $\mu$'s are comparable. Hence, we can see that whether a set is branching recurrent or branching transient is somehow independent of the offspring distribution:

\begin{corollary}
Let $\theta$ be some fixed centered distribution on $\Z^d$ with finite range which is  not supported on a strict subgroup of $\Z^d$. Then for any $K\subseteq\Z^d$, if there exists one nondegenerate critical offspring distribution $\mu$ with finite third moment, such that $K$ is branching recurrent (corresponding to $\mu$ and $\theta$), then for any such offspring distribution, $K$ is branching recurrent (corresponding to that offspring distribution and $\theta$).
\end{corollary}

This paper is organized as follows. In Section 2, we introduce our notations and collect some preliminary knowledge and some results in the previous paper. In Section 3, we introduce branching recurrence and branching transience and establish some basic properties. The remaining sections are devoted to prove Theorem \ref{MT}.

\section{Preliminaries.}
Let us specify the assumptions that will be in force throughout this paper. We always assume that $d\geq 5$ and
\begin{itemize}
\item $\mu$ is a distribution on $\N$ with mean one and finite variance $\sigma^2>0$;
\item $\theta$ is a distribution on $\Z^d$ with mean zero and finite range, which is not supported on a strict subgroup of $\Z^d$.
\end{itemize}
Note that for our main theorem, the Wiener's test, we need to assume further that $\mu$ has finite third moment.

Now we turn to our notations. For a set $K\subseteq\Z^d$, we write $|K|$ for its cardinality. We write $K\ssubset\Z^d$ to express that $K$ is a finite nonempty subset of $\Z^d$. For $x\in\Z^d$ (or $\R^d$), we denote by $|x|$ the Euclidean norm of $x$. We will mainly use the norm $\|\cdot\|$ corresponding the jump distribution $\theta$, i.e. $\|x\|=\sqrt{x\cdot Q^{-1}x}/\sqrt{d}$, where $Q$ is the covariance matrix of $\theta$. For convenience, we set $|0|=\|0\|=0.5$. We denote by $\diam(K)=\sup\{\| a-b\|:a, b\in K\}$, the diameter of $K$ and by $\Rad(K)=\sup\{\| a\| : a\in K\}$, the radius of $K$ respect to $0$. We write $\Ball(r)$ for the ball $\{z\in\Z^d: \| z\|\leq r\}$. For any subsets $A,B$ of $\Z^d$, we denote by $\dist(A,B)=\inf\{ \| x-y \|: x\in A, y\in B\}$ the distance between $A$ and $B$. When $A=\{x\}$ consists of just one point, we just write $\dist(x,B)$ instead.
For any path $\gamma:\{0,\dots,k\}\rightarrow \Z^d$, we let $|\gamma|$ stand for $k$, the length, i.e. the number of edges of $\gamma$ , $\widehat{\gamma}$ for $\gamma(k)$, the endpoint of $\gamma$ and $[\gamma]$ for $k+1$, the number of vertices of $\gamma$. Sometimes we just use a sequence of vertices to express a path. For example, we may write $(\gamma(0),\gamma(1),\dots, \gamma(k))$ for the path $\gamma$. For any $B\subset \Z^d$, we write $\gamma\subseteq B$ to express that all vertices of $\gamma$ except the starting point and the endpoint, lie inside $B$, i.e. $\gamma(i)\in B$ for any $1\leq i\leq k-1$. If the endpoint of a path $\gamma_1:\{0,\dots,|\gamma_1|\}\rightarrow \Z^d$ coincides with the starting point of another path $\gamma_2:\{0,\dots,|\gamma_2|\}\rightarrow \Z^d$, then we can define the composite of $\gamma_1$ and $\gamma_2$ by concatenating $\gamma_1$ and $\gamma_2$:
$$
\gamma_1\circ\gamma_2:\{0,\dots,|\gamma_1|+|\gamma_2|\}\rightarrow \Z^d,
$$
$$
\gamma_1\circ\gamma_2(i)=\left\{
\begin{array}{ll}
\gamma_1(i),&\text{ for } i\leq|\gamma_1|;\\
\gamma_2(i-|\gamma_1|),&\text{ for } i\geq|\gamma_1|.\\
\end{array}
\right.
$$

We now state our convention regarding constants. Throughout the text (unless otherwise specified), we use $C$ and $c$ to denote positive constants depending only on dimension $d$, the critical distribution $\mu$ and the jump distribution $\theta$, which may change from place to place. Dependence of constants on additional parameters will be made or stated explicit. For example, $C(\lambda)$ stands for a positive constant depending on $d,\mu,\theta, \lambda$. For functions $f(x)$ and $g(x)$, we write $f\sim g$ if $\lim_{x\rightarrow \infty}(f(x)/g(x))=1$. We write $f\preceq g$, respectively $f\succeq g$, if there exist constants $C$ such that, $f\leq Cg$, respectively $f\geq Cg$. We use $f\asymp g$ to express that $f\preceq g$ and $f\succeq g$. We write $f\ll g$ for that $\lim_{x\rightarrow \infty}(f(x)/g(x))=0$.

\subsection{Finite and infinite trees.} We are interested in rooted ordered trees (plane trees), in particular, Galton-Watson (GW) trees and its companions. Recall that $\mu=(\mu(i))_{i\in\N}$ is a critical distribution with finite variance $\sigma^2>0$. We exclude the trivial case that $\mu(1)=1$. Throughout this paper, $\mu$ will be fixed. Define another probability measure $\widetilde{\mu}$ on $\N$, call the \textbf{adjoint measure} of $\mu$ by setting $\widetilde{\mu}(i)=\sum_{j=i+1}^\infty \mu(j)$. Since $\mu$ has mean $1$, $\widetilde{\mu}$ is indeed a probability measure. The mean of $\widetilde{\mu}$ is $\sigma^2/2$. A Galton-Watson process with distribution $\mu$ is a process starting with one initial particle, with each particle having independently a random number of children due to $\mu$. The Galton-Watson tree is just the family tree of the Galton-Watson process, rooted at the initial particle. We simply write \text{$\mu$-GW tree} for the Galton-Watson tree with offspring distribution $\mu$.  If we just change the law of the number of children for the root, using $\widetilde{\mu}$ instead of $\mu$ (for other particles still use $\mu$), the new tree is called an \textbf{adjoint $\mu$-GW tree}. The \textbf{infinite $\mu$-GW tree} is constructed in the following way: start with a semi-infinite line of vertices, called the spine, and graft to the left of each vertex in the spine an independent adjoint $\mu$-GW tree, called a bush. The infinite $\mu$-GW tree is rooted at the first vertex of the spine. Here the left means that we assume every vertex in spine except the root is the youngest child (the latest in the Depth-First search order) of its parent. We also need to introduce the so-called \textbf{$\mu$-GW tree conditioned on survival}. Start with a semi-infinite path, called the spine, rooted at the starting point. For each vertex in the spine, with probability $\mu(i+j+1)$ ($i,j\in\N$), it has totally $i+j+1$ children, with exactly $i$ children elder than the child corresponding to the next vertex in the spine, and exactly $j$ children younger. For any vertex not in the spine, it has a random number of children due to $\mu$. The number of children for different vertices are independent. The random tree generated in this way is just the $\mu$-GW tree conditioned on survival. Each tree is ordered using the classical order according to Depth-First search starting from the root. Note that the subtree generated by the vertices of the spine and all vertices on the left of the spine of the $\mu$-GW tree conditioned on survival has the same distribution as the infinite $\mu$-GW tree.

\subsection{Tree-indexed random walk.} Now we introduce the random walk in $\Z^d$ with jump distribution $\theta$, indexed by a random plane tree $T$. First choose some $a\in\Z^d$ as the starting point. Conditionally on $T$ we assign independently to each edge of $T$ a random variable in $\Z^d$ according to $\theta$. Then we can uniquely define a function $\mathcal{S}_T: T\rightarrow \Z^d$, such that, for every vertex $v \in T$ (we also use $T$ for the set of all vertices of the tree $T$), $\mathcal{S}_T(v)-a$ is the sum of the variables of all edges belonging to the unique simple path from the root $o$ to the vertex $u$ (hence $\mathcal{S}_T(o)=a$). A plane tree $T$ together with this random function $\mathcal{S}_T$ is called $T$-indexed random walk starting from $a$. When $T$ is a $\mu$-GW tree, an adjoint $\mu$-GW tree, an infinite $\mu$-GW tree and a $\mu$-GW tree conditioned on survival respectively, we simply call the tree-indexed random walk a \textbf{snake}, an \textbf{adjoint snake}, an \textbf{infinite snake} and an \textbf{incipient infinite snake} respectively. We write $\Snake_x$, $\Snake'_x$, $\Snake^{\infty}_x$ and $\overline{\Snake}^\infty_x$ for a snake, an adjoint snake, an infinite snake, and an incipient infinite snake respectively, starting from $x\in \Z^d$. Note that a snake is just the branching random walk with offspring distribution $\mu$ and jump distribution $\theta$. For an infinite snake, or incipient infinite snake, the random walk indexed by its spine, called its backbone, is just a random walk with jump distribution $\theta$.  Note that all snakes here certainly depend on $\mu$ and $\theta$. But since $\mu$ and $\theta$ are fixed throughout this paper, we omit their dependence in the notation.

Notation: we write $\pS_A(x)$, $\rS_A(x)$, $\qS_A(x)$ and $\overline{\qS}_A(x)$ respectively, for the probability that $\Snake_x$, $\Snake'_x$, $\Snake^{\infty}_x$ and $\overline{\Snake}^\infty_x$ visits $A$.

\subsection{Random walk with killing.} Similar to \cite{Z161}, the main tool we use is the random walk with killing. Suppose that when the random walk is currently at position $x\in\Z^d$, then it is killed, i.e. jumps to a 'cemetery' state $\delta$, with probability $\KRW(x)$, where $\KRW:\Z^d\rightarrow [0,1]$ is a given function. In other words, the random walk with killing rate $\KRW(x)$ (and jump distribution $\theta$) is a Markov chain $\{X_n:n\geq0\}$ on $\Z^d\cup\{\delta\}$ with transition probabilities $p(\cdot,\cdot)$ given by: for $x,y\in\Z^d$,
\begin{equation*}
p(x,\delta)=\KRW(x), \quad p(\delta, \delta)=1, \quad p(x,y)=(1-\KRW(x))\theta(y-x).
\end{equation*}
For any path $\gamma:\{0,\dots, n\}\rightarrow \Z^d$ with length $n$, its probability weight $\BRW(\gamma)$ is defined to be the probability that the path consisting of the first $n$ steps for the random walk with killing starting from $\gamma(0)$ is $\gamma$. Equivalently,
\begin{equation}\label{def-b}
\BRW(\gamma)=\prod_{i=0}^{|\gamma|-1}(1-\KRW(\gamma(i)))\theta(\gamma(i+1)-\gamma(i))
=\SRW(\gamma)\prod_{i=0}^{|\gamma|-1}(1-\KRW(\gamma(i))),
\end{equation}
where $\SRW(\gamma)=\prod_{i=0}^{|\gamma|-1}\theta(\gamma(i+1)-\gamma(i))$ is the probability weight of $\gamma$ corresponding to the random walk with jump distribution $\theta$. In this paper, the killing function is always of the form $\KRW(x)=\KRW_A(x)=\rS_A(x)$ for some set $A\ssubset\Z^d$.

Now we can define the corresponding Green function for $x,y\in \Z^d$:
\begin{equation*}
G_A(x,y)=\sum_{i=0}^{\infty}P(S^{\KRW}_x(n)=y)=\sum_{\gamma:x\rightarrow y}\BRW(\gamma).
\end{equation*}
where $S^{\KRW}_x=(S^{\KRW}_x(n))_{n\in \N}$ is the random walk (with jump distribution $\theta$) starting from $x$, with killing function $\KRW_A$, and the last sum is over all paths from $x$ to $y$. Note that the subscript $A$ indicates that the killing function is $\KRW_A$. For $x\in \Z^d, K\subset \Z^d$, we write $G_A(x,K)$ for $\sum_{y\in K}G_A(x,y)$.

For any $B\subset \Z^d$ and $x, y\in \Z^d$, define the harmonic measure:
$$
\Hm^{B}_A(x,y)= \sum_{\gamma:x\rightarrow y, \gamma\subseteq B}\BRW(\gamma).
$$

Note that when the killing function $\KRW\equiv 0$, the random walk with this killing is just random walk without killing and we will omit the dependence of $\KRW$ in the notation and just write $\Hm^{B}(x,y)$ for example.

We will repeatedly use the following First-Visit Lemma. The idea is to decompose a path according to the first or last visit of a set.
\begin{lemma} \label{bd-1}
For any $B\subseteq \Z^d$ and $a\in B, b\notin B$, we have:
$$
G_A(a,b)=\sum_{z\in B^c}\Hm^{B}_A(a,z)G_\KRW(z,b)=\sum_{z\in B}G_A(a,z)\Hm^{B^c}_A(z,b);
$$
$$
G_A(b,a)=\sum_{z\in B}\Hm^{B^c}_A(b,z)G_\KRW(z,a)=\sum_{z\in  B^c}G_A(b,z)\Hm^{B}_A(z,a).
$$
\end{lemma}

\subsection{Some facts about random walk and the Green function.}
For $x\in\Z^d$, we write $S_x=(S_x(n))_{n\in\N}$ for the random walk with jump distribution $\theta$ starting from $S_x(0)=x$. The norm $\| \cdot \|$ corresponding to $\theta$ for every $x\in \Z^d$ is defined to be $\|x\|=\sqrt{x\cdot Q^{-1}x}/\sqrt{d}$, where $Q$ is the covariance matrix of $\theta$. Note that $\|x\|\asymp |x|$, especially, there exists $c>1$, such that $\Ball(c^{-1}n)\subseteq\BallE(n)\subseteq\Ball(cn)$, for any $n\geq1$. The Green function $g(x,y)$ is defined to be:
$$
g(x,y)=\sum_{n=0}^{\infty}P(S_x(n)=y)=\sum_{\gamma:x\rightarrow y}\SRW(\gamma).
$$
We write $g(x)$ for $g(0,x)$.

Our assumptions about the jump distribution $\theta$ guarantee the standard estimate for the Green function (see e.g. \cite{LL10}):
\begin{equation}\label{green}
g(x)\sim a_d\| x \|^{2-d};
\end{equation}
and
\begin{equation}\label{ngreen}
\sum_{n=0}^{\infty}(n+1)\cdot P(S_0(n)=x)=\sum_{\gamma:0\rightarrow x}[\gamma]\cdot\SRW(\gamma)\asymp \| x\| ^{4-d}\asymp |x|^{4-d}.
\end{equation}
where $a_d=\frac{\mathbf{\Gamma}((d-2)/2)}{2d^{(d-2)/2}\pi^{d/2}\sqrt{\det Q}}$.

The following lemma is natural from the perspective of Brownian motion, the scaling limit of random walk (for a sketch of proof see \cite{Z161}).
\begin{lemma}\label{pro-g1}
Let $U,V$ be two connected bounded open subset of $\R^d$ such that $\overline{U}\subseteq V$. Then there exists a $C=C(U,V)$ such that if $A_n=nU\cap\Z^d, B_n=nV\cap\Z^d$ then when $n$ is sufficiently large,
\begin{equation}\label{pro-b1}
\sum_{\gamma:x\rightarrow y, \gamma\subseteq B_n, |\gamma| \leq2n^2}\SRW(\gamma)\geq Cg(x,y), \text{ for any }x,y\in A_n
\end{equation}
\end{lemma}

\subsection{Some results in the previous paper.}
The key starting point in the previous paper is the following proposition:
\begin{prop}
\begin{equation}\label{p2}
\pS_A(x)=\sum_{\gamma:x\rightarrow A}\BRW(\gamma)=G_A(x,A).
\end{equation}
\end{prop}
From Theorem 1.3 in \cite{Z161}, one can see that (by cutting $A$ into a large number, depending on $\epsilon$, pieces):
\begin{prop}
For any $\epsilon>0$ fixed, when $\dist(x,A)\geq \epsilon \diam (A)$, we have:
\begin{equation}\label{bd-pp}
P(\Snake_x \text{ visits } A)\asymp \frac{\BCap(A)}{(\dist(x,A))^{d-2}},
\end{equation}
Note that the constant in $\asymp$ depends also on $\epsilon$ (independent of $x, A$).
\end{prop}

For $\rS_A(x)$ and $\qS_A(x)$, we have (see Section 8 in \cite{Z161})
\begin{equation}\label{r-p}
\rS_A(x)\asymp \pS_A(x).
\end{equation}
\begin{equation}\label{for-q2}
\qS_A(x)=\sum_{y\in \Z^d}G_A(x,y)\rS_A(y);
\end{equation}
\begin{equation}\label{for-q1}
\qS_A(x)=\sum_{y\in \Z^d}g(x,y)\rS_A(y)\EsC^+_A(y);
\end{equation}
\begin{equation}\label{q-lim}
\qS_A(x)\sim \frac{t_d\cdot a_d^2\sigma^2\BCap(A)}{2\|x\|^{d-4}};
\end{equation}
where $\EsC^+_A(y)$ is the probability that an infinite snake starting from $y$ does not visiting $A$ except possibly for the range of the bush attached to the root (including the root), and $t_d$ is some constant depending on $\theta$. Note that $\EsC^+_A(y)<1-\qS_A(y)\rightarrow 1$, as $y\rightarrow\infty$.

\begin{lemma}\label{bd_Green}
For any $\lambda>0$, there exists $C=C(\lambda)>0$, such that:
for any $A\ssubset \Z^d$ and $x,y\in \Z^d$ with $\|x\|,\|y\|>(1+\lambda)\Rad (A)$, we have:
\begin{equation}\label{bd-Green}
G_A(x,y)\geq C g(x,y).
\end{equation}
\end{lemma}

\section{Branching recurrence and branching transience.}
Analogous to random walk, we give the definitions of branching recurrence and branching transience. Recall that we always assume $d\geq5$.
\begin{definition}
Let $A$ be a subset of $\Z^d$. We call $A$ a \textbf{branching recurrent} (\textbf{B-recurrent}) set if
\begin{equation}\label{d1}
P(\Snake_0^{\infty} \text{ visits } A \text{ infinitely often})=1,
\end{equation}
and a \textbf{branching transient} (\textbf{B-transient}) set if
\begin{equation}\label{d2}
P(\Snake_0^{\infty} \text{ visits } A \text{ infinitely often})=0.
\end{equation}
\end{definition}

In fact, it is equivalent to use the incipient infinite snake in the definition of branching recurrence and branching transience.
\begin{prop}
\begin{equation*}
P(\Snake_0^{\infty} \text{ visits } A \text{ infinitely often})=1\Leftrightarrow
P(\overline{\Snake}_0^{\infty} \text{ visits } A \text{ infinitely often})=1.
\end{equation*}
\end{prop}
\begin{proof}
The necessity is trivial. For the sufficiency, we use the following coupling between $\overline{\Snake}_0^\infty$ and $\Snake_0^{\infty}$. First sample $\overline{\Snake}_0^{\infty}$. Then we can construct $\Snake_0^{\infty}$ as follows: for the the backbone of $\Snake_0^{\infty}$, just use the backbone of $\overline{\Snake}_0^{\infty}$; for each vertex in the backbone, we graft to it an adjoint snake, independently, using either the left adjoint snake or the right one, corresponding to the same vertex in $\overline{\Snake}_0^{\infty}$, with equal probability. When $\overline{\Snake}_0^{\infty}$ visits $A$ infinitely often, there are infinite adjoint snakes on $\overline{\Snake}_0^{\infty}$ visiting $A$. For each vertex on the backbone, we will pick up either the left adjoint snake or the right one independently with equal probability. Therefore, by the strong law of large numbers, an infinite number of adjoint snakes that visits $A$ will be picked up on the process of producing $\Snake_0^{\infty}$, almost surely. It means that $\Snake_0^{\infty}$ visits $A$ infinitely often almost surely.
\end{proof}

\begin{prop}\label{recurrent}
Every set $A\subset\Z^d$ is either B-recurrent or B-transient.
\end{prop}
\begin{proof}
Let $f(x)=P(\Snake_x^{\infty} \text{ visits } A \text{ infinitely often})$. It is easy to see that $f$ is a bounded harmonic function. But every bounded harmonic function in $\Z^d$ is constant. Hence $f\equiv t$ for some $t\in[0,1]$. Let $V$ be the event $\Snake_0^{\infty} \text{ visits } A \text{ infinitely often}$. Since $f\equiv t$, we have $P(V|\mathcal{F}_n)=t$ for any $n$, where $\mathcal{F}_n$ is the $\sigma$-field generated by all 'information' (the tree structure and the random variables corresponding to the edges after $n$-th vertex of the spine). Then $V$ is a tail event, and by the Kolmogorov 0-1 Law, $t$ is either $0$ or $1$.
\end{proof}

If $A$ is finite, since $\qS_A(x)<1$ for large $x$, $f(x)<1$ and must be $0$. Hence we have:

\begin{prop}
Every finite subset of $\Z^d$ is B-transient.
\end{prop}

\section{Inequalities for convolved sums.}
We need the following two inequalities in our proof of Wiener's Test.
\begin{lemma}\label{CVV}
For any $n\in\N^+$, let $B=\Ball(n)$. When $A\subset B$ and $x\in\Z^d$, we have:
\begin{equation}\label{Cvv-q}
\sum_{z\in B}G_A(x,z)\qS_A(z)\preceq (\diam (B))^2 \qS_A(x).
\end{equation}
\begin{equation}\label{Cvv-p}
\sum_{z\in B}G_A(x,z)\pS_A(z)\preceq (\diam (B))^2 \pS_A(x).
\end{equation}
\end{lemma}
We prove \eqref{Cvv-q} here and postpone the proof of \eqref{Cvv-p} until Section 6.
\begin{proof}[Proof of \eqref{Cvv-q}]
For \eqref{Cvv-q}, we do not need to assume that $B$ is a ball and $A\subset B$. In fact, we will prove \eqref{Cvv-q} for any finite subsets $A,B$ of $Z^d$ and $x\in \Z^d$.

We are working at the random walk with killing function $\rS_A$. We consider the following equivalent model: a particle starting from $x$ executes a random walk $S=(S(k))_{k\in\N}$, but at each step, the particle has the probability $\rS_A$ to get a flag (instead of to die) and its movements are unaffected by flags. Let $\tau$ and $\xi$ be the first and last time getting flags (if there is no such time, define both to be infinity). Note that since $\qS_A(z)<1$ (when $|z|$ is large), the total number of flags gained is finite, almost surely. Hence $P(\tau<\infty)=P(\xi<\infty)$. Under this model, one can see that
$$
G_A(x,z)\qS_A(z)=P(\tau<\infty, S(k)=z\text{ for some }k\leq\tau).
$$
Hence it is not more than $\E(\sum_{i=0}^{\tau}\mathbf{1}_{\{S(i)=z\}}; \tau<\infty)$ and the L.H.S. of \eqref{Cvv-q} is not more than
$$
\E(\sum_{i=0}^{\tau}\mathbf{1}_{\{S_x(i)\in B\}}; \tau<\infty)\leq \E(\sum_{i=0}^{\xi}\mathbf{1}_{\{S_x(i)\in B\}}; \xi<\infty).
$$
By considering the place where the particle gets its last flag, one can see:
$$
\E(\sum_{i=0}^{\xi}1_{\{S_x(i)\in B\}}; \xi<\infty)=
\sum_{w\in\Z^d}\left(\sum_{\gamma:x\rightarrow w}
\SRW(\gamma)(\sum_{i=0}^{|\gamma|}\mathbf{1}_{\gamma(i)\in B})\right)\cdot\rS_A(w)\EsC^+_A(w).
$$
We point out a result about random walk and prove it later:
\begin{equation}\label{SRW1}
\sum_{\gamma:x\rightarrow w}
\SRW(\gamma)(\sum_{i=0}^{|\gamma|}\mathbf{1}_{\gamma(i)\in B})\preceq
(\diam(B))^2 \sum_{\gamma:x\rightarrow w}\SRW(\gamma).
\end{equation}
Hence we get:
\begin{align*}
\sum_{z\in B}G_A(x,z)\qS_A(z)\preceq&(\diam(B))^2\sum_{w\in\Z^d}\sum_{\gamma:x\rightarrow w}\SRW(\gamma)\rS_A(w)\EsC^+_A(w)\\
=&(\diam(B))^2\sum_{w\in\Z^d}g(x,w)\rS_A(w)\EsC^+_A(w)\\
\stackrel{\eqref{for-q1}}{=}&(\diam(B))^2\qS_A(x).\\
\end{align*}
Now we just need to prove \eqref{SRW1}. First we assume $x,w\in B$, then
\begin{multline*}
\sum_{\gamma:x\rightarrow w}\SRW(\gamma)(\sum_{i=0}^{|\gamma|}\mathbf{1}_{\gamma(i)\in B})
\leq \sum_{\gamma:x\rightarrow w}\SRW(\gamma)[\gamma]\stackrel{\eqref{ngreen}}{\asymp} |x-w|^{4-d}\\
\leq (\diam(B))^2|x-w|^{2-d} \asymp(\diam(B))^2 g(x,w)=(\diam(B))^2\sum_{\gamma:x\rightarrow w}\SRW(\gamma).
\end{multline*}
For general $x,w$, one just need to decompose $\gamma$ into three paths according to the first and last visiting time of $B$. For example, when $x,w\notin B$, we have:
\begin{align*}
\sum_{\gamma:x\rightarrow w}&\SRW(\gamma)(\sum_{i=0}^{|\gamma|}\mathbf{1}_{\gamma(i)\in B})
=\sum_{y,z\in B}\Hm^{B^c}(x,y)\left(\sum_{\gamma':y\rightarrow z}\SRW(\gamma')\sum_{i=0}^{|\gamma'|}\mathbf{1}_{\gamma'(i)\in B}\right)\Hm^{B^c}(z,w)\\
\preceq&\sum_{y,z\in B}\Hm^{B^c}(x,y)\left((\diam(B))^2
\sum_{\gamma':y\rightarrow z}\SRW(\gamma')\right)\Hm^{B^c}(z,w)\\
=&(\diam(B))^2\sum_{y,z\in B}\Hm^{B^c}(x,y)\sum_{\gamma':y\rightarrow z}\SRW(\gamma')\Hm^{B^c}(z,w)\\
=&(\diam(B))^2\sum_{\gamma:x\rightarrow w, \gamma \text{ visits }B}\SRW(\gamma)
\leq(\diam(B))^2\sum_{\gamma:x\rightarrow w}\SRW(\gamma).\\
\end{align*}
When just one of $x$ and $w$ is in $B$, the proof is similar but easier.
\end{proof}

\section{Restriction lemmas.}
Recall that we have:
$$
\pS_A(x)=\sum_{\gamma:x\rightarrow A}\BRW(\gamma).
$$
Our goals of this section are to show:
\begin{prop}\label{pS_form}
For any $n\in\N^+$ sufficiently large and $A\subset \Ball(n),x\in \Ball(n)$, we have:
\begin{equation}\label{pS-form}
\pS_A(x)\asymp\sum_{\gamma:x\rightarrow A,\gamma\subseteq\Ball(1.1n)}\BRW(\gamma).
\end{equation}
\end{prop}
\begin{prop}\label{qS_form}
For any $n\in\N^+$ sufficiently large and $A\subset \Ball(n),x\in \Ball(n)$, we have:
\begin{equation}\label{qS-form}
\qS_A(x)\asymp \sum_{\gamma:x\rightarrow A,\gamma\subseteq\Ball(4n)}[\gamma]\cdot\BRW(\gamma).
\end{equation}
\end{prop}

We first introduce some notations. Since $\theta$ has finite range, we can define the outer boundary $\outBd B$ for any $B\subseteq \Z^d$ by
$$
\outBd B =\{z\in \Z^d\setminus B: \exists y\in B, \theta(z-y)\vee \theta(y-z)>0\}.
$$
Note that for any $y\in \outBd B$, $\dist(y, B)$ is bounded above by a constant depending on $\theta$.
For $A\subset B\subset \Z^d$ and $x,y \in B\cup \outBd B$, write
$$
G_A^B(x,y)=\sum_{\gamma:x\rightarrow y,\gamma\subseteq B}\BRW(\gamma).
$$

\begin{lemma}
For any $\lambda_1,\lambda_2,\lambda_3>0$, there exists $C=C(\lambda_1,\lambda_2,\lambda_3)>0$ satisfying the following. When $n$ is sufficiently large,
let $B_0=\Ball(n)$,$B_1=\Ball ((1+\lambda_1)n)$, $B_2=\Ball((1+\lambda_1+\lambda_2)n)$ and $B=\Ball((1+\lambda_1+\lambda_2+\lambda_3)n)$. Then for any $x,y\in B_2\setminus B_1$ and $A\subset B_0$, we have:
\begin{equation}\label{t1}
G_A^{B}(x,y)\geq C G_A(x,y).
\end{equation}
\end{lemma}
\begin{proof}
Let $B'=\Ball ((1+\lambda_1/2)n)$. Note that for any $y\in B\setminus B'$, by \eqref{bd-pp} and \eqref{r-p}, we have $\rS_A(y)\asymp\pS_A(y)\preceq n^{-2}$. Hence, we have: for any $\gamma\subseteq B\setminus B'$ with $|\gamma|\leq 2n^2$, $\BRW(\gamma)/\SRW(\gamma)\geq(1-c/n^2)^{2n^2}\succeq1$.

Therefore, by Lemma \ref{pro-g1}, one can see that:
\begin{align*}
G_A^B(x,y)&=\sum_{\gamma:x\rightarrow y,\gamma\subseteq B}\BRW(\gamma)
\geq\sum_{\gamma:x\rightarrow y, \gamma\subseteq B\setminus B_1, |\gamma| \leq2n^2}\BRW(\gamma)\\
&\succeq\sum_{\gamma:x\rightarrow y, \gamma\subseteq B\setminus B_1, |\gamma| \leq2n^2}\SRW(\gamma)\asymp g(x,y)\geq G_A(x,y).
\end{align*}
\end{proof}
\begin{lemma}\label{lem-restriction}
For any $\lambda>0,\iota>0$, there exists $C=C(\lambda,\iota)>0$ satisfying the following. When $n$ is sufficiently large,
let $B_0=\Ball(n)$, $B_1=\Ball (1+\lambda)n$, $B=\Ball(1+\lambda+\iota)n$. Then for any $x,y\in B_1$ and $A\subset B_0$, we have:
\begin{equation}\label{t2}
G_A^B(x,y)\geq C G_A(x,y).
\end{equation}
\end{lemma}
\begin{proof}
By the last lemma, one can get, for any $z,w\in\outBd B_1$,
$$
G_A^B(z,w)\succeq G_A(z,w).
$$
For any $x,y\in B_1$, we have:
$$
G_A^B(x,y)=G_A^{B_1}(x,y)+\sum_{\gamma:x\rightarrow y, \gamma \text{ visits } B_1^c, \gamma\subseteq B}\BRW(\gamma).
$$
By considering the first and last visit in $B_1^c$, we have:
\begin{align*}
&\sum_{\gamma:x\rightarrow y, \gamma \text{ visits } B_1^c, \gamma\subseteq B}\BRW(\gamma)=
\sum_{z,w\in\outBd B_1}\Hm^{B_1}_A(x,z)G_A^B(z,w)\Hm^{B_1}_A(w,y)\\
&\;\stackrel{\eqref{t1}}{\succeq}\quad\quad\sum_{z,w\in\outBd B_1}\Hm^{B_1}_A(x,z)G_A(z,w)\Hm^{B_1}_A(w,y)=\sum_{\gamma:x\rightarrow y, \gamma \text{ visits } B_1^c}\BRW(\gamma).\\
\end{align*}
Hence, we have:
\begin{align*}
G_A^B(x,y)=&G_A^{B_1}(x,y)+\sum_{\gamma:x\rightarrow y, \gamma \text{ visits } B_1^c, \gamma\subseteq B}\BRW(\gamma)
\\
\succeq& G_A^{B_1}(x,y)+\sum_{\gamma:x\rightarrow y, \gamma \text{ visits } B_1^c}\BRW(\gamma)=G_A(x,y).
\end{align*}
\end{proof}
Now we can show Proposition \ref{pS_form}:
\begin{proof}[Proof of Proposition \ref{pS_form}]
Let $B=\Ball(1.1n)$. We have:
$$
\pS_A(x)=\sum_{\gamma:x\rightarrow A}\BRW(\gamma)=\sum_{z\in A}G_A(x,z)
\stackrel{\eqref{t2}}{\asymp}\sum_{z\in A}G^B_A(x,z)=\sum_{\gamma:x\rightarrow A,\gamma\subseteq B}\BRW(\gamma).
$$
\end{proof}

Now we turn to $\qS_A(x)$. The starting point is:
\begin{lemma}
For any $a\in\Z^d$, $A\ssubset\Z^d,B\subset \Z^d$, we have:
\begin{equation}\label{Cvv-pB}
\sum_{z\in B}G_A(x,z)\pS_A(z)=\sum_{\gamma:x\rightarrow A}\BRW(\gamma)
\sum_{i=0}^{|\gamma|}\mathbf{1}_{\gamma(i)\in B}.
\end{equation}
\end{lemma}
\begin{proof}
\begin{align*}
\sum_{z\in B}G_A(x,z)\pS_A(z)
=&\sum_{z\in B}\sum_{\gamma_1:x\rightarrow z}\BRW(\gamma_1)\sum_{\gamma_2:z\rightarrow A}\BRW(\gamma_2)\\
&=\sum_{z\in B}\sum_{\gamma_1:x\rightarrow z}\sum_{\gamma_2:z\rightarrow A}\BRW(\gamma_1)\BRW(\gamma_2)\\
&=\sum_{z\in B}\sum_{\gamma_1:x\rightarrow z}\sum_{\gamma_2:z\rightarrow A}\BRW(\gamma_1\circ\gamma_2)\\
&=\sum_{\gamma:x\rightarrow A}\BRW(\gamma)\cdot\sum_{i=0}^{|\gamma|}\mathbf{1}_{\gamma(i)\in B}.\\
\end{align*}
The last equality is due to the fact that for any $\gamma:x\rightarrow A$, there are exactly $\sum_{i=0}^{|\gamma|}\mathbf{1}_{\gamma(i)\in B}$ ways to decompose $\gamma$ to the composite of two paths $\gamma_1$ and $\gamma_2$ such that the common point of $\gamma_1$ and $\gamma_2$ is in $B$.
\end{proof}

\begin{corollary}
\begin{equation}\label{q-gamma}
\qS_A(x)\asymp \sum_{\gamma:x\rightarrow A}[\gamma]\cdot\BRW(\gamma).
\end{equation}
\end{corollary}
\begin{proof}
By \eqref{for-q2} and \eqref{r-p}, we have:
\begin{equation}\label{q=cvv-p}
\qS_A(x)\asymp \sum_{z\in\Z^d}G_A(x,z)\pS_A(z).
\end{equation}
By the last lemma, we have $\sum_{z\in\Z^d}G_A(x,z)\pS_A(z)=\sum_{\gamma:x\rightarrow A}[\gamma]\cdot\BRW(\gamma)$.
\end{proof}

\begin{lemma}\label{qS_form0}
For any $n$ sufficiently large, $A\subset \Ball(n), x\in \Z^d$ with $\|x\|\geq 1.1n$, we have:
\begin{equation}\label{qS-form0}
\qS_A(x)\asymp \sum_{\gamma:x\rightarrow A,\gamma\subseteq\Ball(3\|x\|)}[\gamma]\cdot\BRW(\gamma).
\end{equation}
\end{lemma}
\begin{proof}
The part of '$\succeq$' is trivial by the last corollary. Hence we only need to prove the other part.
As in the last corollary, we have:
\begin{equation*}
\qS_A(x)\asymp \sum_{z\in\Z^d}G_A(x,z)\pS_A(z).
\end{equation*}
First by \eqref{green},\eqref{bd-pp} and \eqref{bd-Green}, one can see that:
\begin{align*}
\sum_{z\in\Ball(2\|x\|)\setminus\Ball(1.5\|x\|)}G_A(x,z)\pS_A(z)&\asymp \sum_{z\in\Ball(2\|x\|)\setminus\Ball(1.5\|x\|)} g(x,z)\frac{\BCap(A)}{(\dist(z,A))^{d-2}}\\
&\asymp \sum_{z\in\Ball(2\|x\|)\setminus\Ball(1.5\|x\|)} \frac{1}{|x|^{d-2}}\frac{\BCap(A)}{|x|^{d-2}}\\
&\asymp |x|^{d}\frac{1}{|x|^{d-2}}\frac{\BCap(A)}{|x|^{d-2}}= \frac{\BCap(A)}{|x|^{d-4}}.\\
\end{align*}
Similarly, we can get:
\begin{align*}
\sum_{z\in\Ball(2\|x\|)^c}G_A(x,z)\pS_A(z)&\asymp \sum_{z\in\Ball(2\|x\|)^c} g(x,z)\frac{\BCap(A)}{(\dist(z,A))^{d-2}}\\
&\asymp \sum_{z\in\Ball(2\|x\|)^c} \frac{1}{|z|^{d-2}}\frac{\BCap(A)}{|z|^{d-2}}\\
&\asymp \BCap(A)\sum_{z\in\Ball(2\|x\|)^c} \frac{1}{|z|^{2d-4}}\asymp \frac{\BCap(A)}{|x|^{d-4}}.\\
\end{align*}
Hence we have: $$
\qS_A(x)\asymp\sum_{z\in\Z^d}G_A(x,z)\pS_A(z)\asymp\sum_{z\in\Ball(2\|x\|)}G_A(x,z)\pS_A(z).
$$
By Lemma \ref{lem-restriction}, we have (let $B=\Ball(3\|x\|)$):
\begin{align*}
\sum_{z\in\Ball(2\|x\|)}G_A(x,z)\pS_A(z)&=\sum_{z\in\Ball(2\|x\|)}G_A(x,z)\sum_{y\in A}G_A(z,y)\\
&\asymp\sum_{z\in\Ball(2\|x\|)}G_A^B(x,z)\sum_{y\in A}G^B_A(z,y)\\
&=\sum_{z\in\Ball(2\|x\|)}\sum_{\gamma_1:x\rightarrow z,\gamma_1\subseteq B}\BRW(\gamma_1)\sum_{\gamma_2:z\rightarrow A,\gamma_2\subseteq B}\BRW(\gamma_2)\\
&=\sum_{z\in\Ball(2\|x\|)}\sum_{\gamma_1:x\rightarrow z,\gamma_1\subseteq B}\sum_{\gamma_2:z\rightarrow A,\gamma_2\subseteq B}\BRW(\gamma_1\circ\gamma_2)\\
&\leq \sum_{\gamma:x\rightarrow A,\gamma\subseteq B}[\gamma]\BRW(\gamma).\\
\end{align*}
This completes the proof.
\end{proof}

\begin{proof}[Proof of Proposition \ref{qS_form}]
Let $B=\Ball(1.1n)$ and $B'=\Ball(4n)$. We have:
\begin{equation*}
\qS_A(x)\asymp\sum_{\gamma:x\rightarrow A}[\gamma]\BRW(\gamma)
=\sum_{\gamma:x\rightarrow A, \gamma\subseteq B}[\gamma]\BRW(\gamma)
+\sum_{\gamma:x\rightarrow A, \gamma\text{ visits } B^c}[\gamma]\BRW(\gamma).
\end{equation*}
By considering the first visit of $B^c$, the second term is equal to:
\begin{align*}
&\sum_{y\in \outBd B}\sum_{\gamma_1:x\rightarrow y, \gamma_1\subseteq B}
\sum_{\gamma_2:y\rightarrow A}(|\gamma_1|+[\gamma_2])(\BRW(\gamma_1)\BRW(\gamma_2))=\\
&\sum_{y\in \outBd B}\sum_{\gamma_1:x\rightarrow y, \gamma_1\subseteq B}|\gamma_1|\BRW(\gamma_1)
\sum_{\gamma_2:y\rightarrow A}\BRW(\gamma_2)+\sum_{y\in \outBd B}
\sum_{\gamma_1:x\rightarrow y, \gamma_1\subseteq B}\BRW(\gamma_1)
\sum_{\gamma_2:y\rightarrow A}[\gamma_2]\BRW(\gamma_2)\\
&\stackrel{\eqref{p2}\eqref{q-gamma}}{\asymp}\sum_{y\in \outBd B}\sum_{\gamma_1:x\rightarrow y, \gamma_1\subseteq B}|\gamma_1|\BRW(\gamma_1)
\pS_A(y)+\sum_{y\in \outBd B}
\sum_{\gamma_1:x\rightarrow y, \gamma_1\subseteq B}\BRW(\gamma_1)\qS_A(y)\\
&\stackrel{\eqref{qS-form0},\eqref{pS-form}}{\asymp}\sum_{y\in \outBd B}\sum_{\gamma_1:x\rightarrow y, \gamma_1\subseteq B}|\gamma_1|\BRW(\gamma_1)
\sum_{\gamma_2:y\rightarrow A,\gamma_2\subseteq B'}\BRW(\gamma_2)\\
&\quad \quad\quad \quad\quad\quad\quad\quad+\sum_{y\in \outBd B}\sum_{\gamma_1:x\rightarrow y, \gamma_1\subseteq B}\BRW(\gamma_1)
\sum_{\gamma_2:y\rightarrow A,\gamma_2\subseteq B'}[\gamma_2]\BRW(\gamma_2)\\
&=\sum_{y\in \outBd B}\sum_{\gamma_1:x\rightarrow y, \gamma_1\subseteq B}\sum_{\gamma_2:y\rightarrow A,\gamma_2\subseteq B'}(|\gamma_1|+[\gamma_2])(\BRW(\gamma_1)\BRW(\gamma_2))\\
&= \sum_{\gamma:x\rightarrow A, \gamma\text{ visits } B^c,\gamma\subseteq B'}[\gamma]\BRW(\gamma).\\
\end{align*}
Hence, we get
\begin{equation*}
\qS_A(x)\asymp\sum_{\gamma:x\rightarrow A, \gamma\subseteq B}[\gamma]\BRW(\gamma)
+\sum_{\gamma:x\rightarrow A, \gamma\text{ visits } B^c,\gamma\subseteq B'}[\gamma]\BRW(\gamma)=\sum_{\gamma:x\rightarrow A, \gamma\subseteq B'}[\gamma]\BRW(\gamma).
\end{equation*}
This completes the proof.
\end{proof}

\section{Visiting probability by an infinite snake.}
In this section we establish the following bounds analogous to \eqref{bd-pp}:
\begin{theorem}\label{bd-infinite}
For any $A\ssubset \Z^d$ and $x\in\Z^d$ with $\|x\|\geq 2\Rad(A)$, we have:
\begin{equation}\label{bd_infinite}
\qS_A(x)\asymp\frac{\BCap(A)}{(\dist(x,A))^{d-4}}.
\end{equation}
\end{theorem}
\begin{remark}
Similarly to \eqref{bd-pp}, by cutting $A$ into small pieces, one can replace $\|x\|\geq 2\Rad(A)$ by $\dist(x,A)\geq \epsilon \diam(A)$, for any $\epsilon>0$.
\end{remark}
\begin{proof}
It suffices to show the case when $\Rad (A)$ is sufficiently large since we know the asymptotical behavior when $x$ is far away (\eqref{q-lim}). The part for $\succeq$ is straightforward and similar to the first part of the proof of Lemma \ref{qS_form0}:
\begin{align*}
\qS_A(x)&\stackrel{\eqref{q=cvv-p}}{\asymp}\sum_{z\in\Z^d}G_A(x,z)\pS_A(z)\geq\sum_{2\|x\|\leq\|z\|\leq4\|x\|}G_A(x,z)\pS_A(z)\\
\stackrel{\eqref{bd-pp}\eqref{bd-Green}}{\asymp} &\sum_{2\|x\|\leq\|z\|\leq4\|x\|}\frac{1}{|x-z|^{d-2}}\frac{\BCap(A)}{(\dist(z,A))^{d-2}}
\asymp \sum_{2\|x\|\leq\|z\|\leq4\|x\|}\frac{1}{|z|^{d-2}}\frac{\BCap(A)}{|z|^{d-2}}\\
\asymp&|x|^{d}\frac{1}{|x|^{d-2}}\frac{\BCap(A)}{|x|^{d-2}}=\frac{\BCap(A)}{|x|^{d-4}}
\asymp \frac{\BCap(A)}{(\dist(x,A))^{d-4}}.\\
\end{align*}
The other part can be implied by \eqref{bd-pp} and the following lemma (let $n=\|x\|$).
\end{proof}
\begin{lemma}
For any $n\in\N^+$ sufficiently large, $A\subset\Ball(n), y\in\Ball(n)$, we have:
\begin{equation}\label{q-p}
\qS_A(y)\preceq n^2 \pS_A(y).
\end{equation}
\end{lemma}
\begin{proof}
Let $B=\Ball(4n)$. By \eqref{qS-form} and \eqref{p2}, it suffices to prove:
\begin{equation}\label{ob}
\sum_{\gamma:y\rightarrow A,\gamma\subseteq B} [\gamma]\BRW(\gamma)\preceq
n^2\sum_{\gamma:y\rightarrow A,\gamma\subseteq B}\BRW(\gamma).
\end{equation}
By \eqref{Cvv-q} and \eqref{qS-form}, one can get:
\begin{equation}
\sum_{z\in B}G_A(y,z)\qS_A(z) \preceq
n^2\sum_{\gamma:y\rightarrow A,\gamma\subseteq B}[\gamma]\BRW(\gamma).
\end{equation}
For the left hand side, we have:
\begin{align*}
&\sum_{z\in B}G_A(y,z)\qS_A(z)\stackrel{\eqref{q-gamma}}{\asymp}
\sum_{z\in B}\sum_{\gamma_1:y\rightarrow z}\BRW(\gamma_1)
\sum_{\gamma_2:z\rightarrow A}[\gamma_2]\BRW(\gamma_2)\\
\geq&\sum_{z\in B}\sum_{\gamma_1:y\rightarrow z,\gamma_1\subseteq B}
\sum_{\gamma_2:z\rightarrow A,\gamma_2\subseteq B}[\gamma_2]\BRW(\gamma_1\circ\gamma_2)\\
=&\sum_{\gamma:y\rightarrow A,\gamma\subseteq B}(1+2+...+[\gamma])\BRW(\gamma)\asymp\sum_{\gamma:y\rightarrow A,\gamma\subseteq B}[\gamma]^2\BRW(\gamma).\\
\end{align*}
Hence, we have:
\begin{equation}\label{a3}
\sum_{\gamma:y\rightarrow A,\gamma\subseteq B}[\gamma]^2\BRW(\gamma)\preceq
n^2\sum_{\gamma:y\rightarrow A,\gamma\subseteq B}[\gamma]\BRW(\gamma).
\end{equation}
By Cauchy-Schwartz inequality:
\begin{align*}
\left(\sum_{\gamma:y\rightarrow A,\gamma\subseteq B} [\gamma]\BRW(\gamma)\right)^2&\leq
\left(\sum_{\gamma:y\rightarrow A,\gamma\subseteq B} [\gamma]^2\BRW(\gamma)\right)\cdot
\left(\sum_{\gamma:y\rightarrow A,\gamma\subseteq B} \BRW(\gamma)\right)\\
&\preceq n^2\sum_{\gamma:y\rightarrow A,\gamma\subseteq B}[\gamma]\BRW(\gamma)\cdot
\left(\sum_{\gamma:y\rightarrow A,\gamma\subseteq B} \BRW(\gamma)\right).\\
\end{align*}
Then \eqref{ob} follows and we complete the proof.
\end{proof}
\begin{proof}[Proof of \eqref{Cvv-p}]
When $x\in B$, by the last lemma (recall that $\qS_A(x)\asymp \sum_{z\in\Z^d}G_A(x,z)\pS_A(x)$),
we have the desired bound.
Now we assume $x\notin B$. By considering the first visit of $B$, we have
\begin{align*}
\sum_{z\in B}&G_A(x,z)\pS_A(z)
\stackrel{\eqref{Cvv-pB}}{=}\sum_{\gamma:x\rightarrow A}(\sum_{i=0}^{|\gamma|}\mathbf{1}_{\gamma(i)\in B})\BRW(\gamma)\\
=&\sum_{y \in B}\sum_{\gamma_1:x\rightarrow y,\gamma_1\subseteq B^c}
\sum_{\gamma_2:y\rightarrow A}\BRW(\gamma_1\circ \gamma_2)
(\sum_{i=0}^{|\gamma_2|}\mathbf{1}_{\gamma_2(i)\in B})\\
=&\sum_{y \in B}\sum_{\gamma_1:x\rightarrow y,\gamma_1\subseteq B^c}\BRW(\gamma_1)
\sum_{\gamma_2:y\rightarrow A}(\sum_{i=0}^{|\gamma_2|}\mathbf{1}_{\gamma_2(i)\in B})\BRW(\gamma_2)\\
\stackrel{\eqref{Cvv-pB}}{=}&\sum_{y \in B}\sum_{\gamma_1:x\rightarrow y,\gamma_1\subseteq B^c}\BRW(\gamma_1)
\sum_{z\in B}G_A(y,z)\pS_A(z)\\
\stackrel{(\ast)}{\preceq}&\sum_{y \in B}\sum_{\gamma_1:x\rightarrow y,\gamma_1\subseteq B^c}\BRW(\gamma_1) (\diam(B))^2 \pS_A(y)\\
=&(\diam(B))^2\sum_{y \in B}\sum_{\gamma_1:x\rightarrow y,\gamma_1\subseteq B^c}\BRW(\gamma_1) \pS_A(y)\\
=&(\diam(B))^2 \pS_A(x).\\
\end{align*}
$(\ast)$ is because we have proved that \eqref{Cvv-p} is true for $x\in B$ and for the last line, we use the First-Visit Lemma and \eqref{p2}.
\end{proof}

\section{Upper bounds for the probabilities of visiting two sets.}
In this section we aim to prove the following inequalities which we will use in the proof of Wiener's Test.
\begin{lemma}
For any disjoint nonempty subsets $A, B\ssubset\Z^d$ and $x\in \Z^d$, we have:
\begin{equation}\label{two1}
P(\Snake_x\text{ visits both }A \& B)\preceq \sum_{z\in\Z^d}G_{A\cup B}(x,z)\pS_A(z)\pS_B(z);
\end{equation}
\begin{multline}\label{two2}
P(\Snake_x^\infty\text{ visits both }A \& B)\preceq \\
\sum_{z\in\Z^d}G_{A\cup B}(x,z)
\left(\pS_A(z)\qS_B(z)+\qS_A(z)\pS_B(z)+P(\Snake'_z\text{ visits both }A \& B)\right).
\end{multline}
\end{lemma}
\begin{proof}
\eqref{two2} is a bit easier and we prove it first. When an infinite snake $\Snake_x^\infty=(T,\Snake_T)$ visits both $A$ and $B$, let $u$ be the first vertex in the spine such that the image of the bush graft to $u$ under $\Snake_T$ intersects $A\cup B$. Assume $(v_0,\dots, v_{k})$ is the unique simple path in the spine from $o$ to $u$. Define $\bb_{(A,B)}(\Snake_x^\infty)=(\Snake_T(v_0),\dots,\Snake_T(v_{k}))$. For any path $\gamma=(\gamma(0),\dots,\gamma(k))$ starting from $x$ with length $|\gamma|=k$, we would like to estimate $P(\bb_{(A,B)}(\Snake_x^\infty)=\gamma)$. If we can show that:
\begin{multline}\label{temp1}
P(\bb_{(A,B)}(\Snake_x^\infty)=\gamma)\preceq\\
\BRW(\gamma)\left(\pS_A(\widehat{\gamma})\qS_B(\widehat{\gamma})+\qS_A(\widehat{\gamma})\pS_B(\widehat{\gamma})
+P(\Snake'_{\widehat{\gamma}}\text{ visits both }A \& B)\right),
\end{multline}
then by summation, one can get \eqref{two2}.

Now we argue that \eqref{temp1} is correct. Let $\mathbf{t}$ be the bush grafted to $u$. There are three possibilities: $\Snake_T(\mathbf{t})$ visits $A$ but not $B$, visits $B$ but not $A$ or visits both $A$ and $B$. For the first one, to guarantee $\bb_{(A,B)}(\Snake_x^\infty)=\gamma$, we need three conditions to be true. The first is that $\Snake_T$ maps $(v_0,\dots, v_{k})$ to $\gamma$ and that the image of each bush grafted to $v_i$ does not intersect $A\cup B$, for $i=0,\dots,k-1$. The probability of this condition being true is $\BRW(\gamma)$. The second condition is that $\Snake_T(\mathbf{t})$ intersects $A$ but not $B$. The probability of this condition being true is at most $\rS_A(\widehat{\gamma})\asymp\pS_A(\widehat{\gamma})$. The last condition is that the image of the bushes after $u$ intersects $B$. The probability of this condition being true is at most $\qS_B(\widehat{\gamma})$. Note that for fixed $\gamma$, the three conditions are independent. Hence we have:
$$
P(\bb_{(A,B)}(\Snake_x^\infty)=\gamma, \Snake_T(\mathbf{t})\text{ visits } A \text{ not } B)\leq
\BRW(\gamma)\pS_A(\widehat{\gamma})\qS_B(\widehat{\gamma}).
$$
Similarly, one can get the other two inequalities. This completes the proof of \eqref{two2}.

For \eqref{two1}, we use a similar idea. When a snake $\Snake_x=(T,\Snake_T)$ visits both $A$ and $B$, then $V_A\DeFine\{v\in T: \Snake_T(v)\in A\}$ and $V_B\DeFine\{v\in T: \Snake_T(v)\in B\}$ are nonempty. We call a vertex $v\in T$ good, if $v$ is the last common ancestor for some $u_1\in V_A$ and $u_2\in V_B$ (any vertex is regarded as an ancestor of itself). Since for any $u_1\in V_A$ and $u_2\in V_B$, they have the unique last common ancestor. Hence there exists at least one good vertex and we choose the first good one (due to the default order, Depth-First order), say $u$. Assume $\gamma=(v_0,\dots,v_k)$ is the unique simple path in $T$ from the root $o$ to $u$. Define $\bb_{(A,B)}(\Snake_x)=(\Snake_x(v_0),\dots,\Snake_x(v_k))$. As before, we would like to estimate $P(\bb_{(A,B)}(\Snake_x)=\gamma)$, for a fixed path $\gamma=(\gamma(0)),\dots,\gamma(k)$ starting from $x$, with length $|\gamma|=k$. We argue that:
\begin{equation}\label{tt}
P(\bb_{(A,B)}(\Snake_x)=\gamma)\preceq \BRW(\gamma)\pS_A(\widehat{\gamma})\pS_B(\widehat{\gamma}).
\end{equation}

Since $u$ is the first good vertex, one can see that all vertices in $V_A\cup V_B$ are descendants of $u$ or $u$ itself. In particular,
\begin{align}\label{con1}
\mathrm{any~vertex~before~}u\mathrm{~is~not~in~} V_A\cup V_B.
\end{align}
Here, 'before' is due to the Depth-First search order.
This is the first necessary condition for the event $\bb_{(A,B)}(\Snake_x)=\gamma$ being true. Similar to the computations in Section 3.2, the probability for \eqref{con1} being true is $\BRW(\gamma)$. Note that this condition just depends on $(T\setminus T_{u}, \Snake_T|_{T\setminus T_{u}})$, where $T_{u}$ is the subtrees generated by $u$ and its descendants, and $T\setminus T_{u}$ is the tree generated by $u$ and those vertices outside $T_{u}$.

On the other hand, since $u$ is the last common ancestor for some $u_1\in V_A$ and $u_2\in V_B$, when $u\notin V_A\cup V_B$, $u$ must have two different children $u^1$ and $u^2$, such that $\Snake_T(T_{u^1})\cap A\neq \emptyset$ and $\Snake_T(T_{u^2})\cap B\neq \emptyset$. This is the second necessary condition for the event $\bb_{(A,B)}(\Snake_x)=\gamma$ being true. Note that for fixed $\gamma$, this condition is independent of \eqref{con1}, and its probability is at most
\begin{align*}
\sum_{n=2}^{\infty}\mu(n)n(n-1)\pS_A(\widehat{\gamma})\pS_B(\widehat{\gamma})
=\sigma^2\pS_A(\widehat{\gamma})\pS_B(\widehat{\gamma}).
\end{align*}
When $u\in V_A\cup V_B$, say $\in A$, then similarly, $u$ must have a descendant mapped into $B$. The probability for this condition is: $\pS_B(\widehat{\gamma})=\pS_A(\widehat{\gamma})\pS_B(\widehat{\gamma})$.
Combining the two conditions one can get \eqref{tt}. By summation, one can get \eqref{two1}. This completes the proof of \eqref{two1}.
\end{proof}

We require the assumption of the finite third moment of $\mu$ only for the following lemma.
\begin{lemma}\label{3-moment}
When $\mu$ has finite third moment, we have:
\begin{equation}\label{adjoint}
P(\Snake_x\text{ visits both }A \& B)\asymp P(\Snake'_x\text{ visits both }A \& B).
\end{equation}
\end{lemma}
\begin{proof}
In fact, we will show:
\begin{equation}\label{ineq-1}
P(\Snake_x\text{ visits both }A \& B)\asymp\pS_A(x)\pS_B(x)+P(\overline{\Snake}_x\text{ visits both }A \& B);
\end{equation}
\begin{equation}\label{ineq-2}
P(\Snake'_x\text{ visits both }A \& B)\asymp\pS_A(x)\pS_B(x)+P(\overline{\Snake}_x\text{ visits both }A \& B);
\end{equation}
where $\overline{\Snake}_x$ is the finite snake from $x$ conditioned on the initial particle having only one child.

For the upper bound of \eqref{ineq-1}, consider whether $\Snake_x$ visits $A$ via the same child of the initial particle as it visits $B$ via. If it does, this probability is at most
$$
\sum_{i=1}^{\infty}\mu(i)\cdot i P(\overline{\Snake}_x\text{ visits both }A \& B)
=\E(\mu)P(\overline{\Snake}_x\text{ visits both }A \& B).
$$
If it does not, this probability is at most
$$
\sum_{i=2}^{\infty}\mu(i)\cdot i(i-1) P(\overline{\Snake}_x\text{ visits }A)P(\overline{\Snake}_x\text{ visits }B)
\asymp \pS_A(x)\pS_B(x).
$$
Note that we use the fact that $\sum_{i=2}^{\infty}\mu(i)\cdot i(i-1)$ is bounded by the second moment of $\mu$ and $P(\overline{\Snake}_x\text{ visits }A)\asymp \pS_A(x)$, which can be proved similar to \eqref{r-p} (or see (8.5) in \cite{Z161}). Combining the last two inequalities, we get the upper bound of \eqref{ineq-1}.

For the lower bound, it is easy to see that
$$
P(\Snake_x\text{ visits both }A \& B)\geq \left(\sum_{i\geq2}\mu(i)\right)P(\overline{\Snake}_x\text{ visits }A)P(\overline{\Snake}_x\text{ visits }B)\asymp \pS_A(x)\pS_B(x);
$$
$$
P(\Snake_x\text{ visits both }A \& B)\geq \left(\sum_{i\geq1}\mu(i)\right)P(\overline{\Snake}_x\text{ visits both }A \& B).
$$
Combining these two, we can the lower bound of \eqref{ineq-1}.

Similarly one can get \eqref{ineq-2}. Note that for the upper bound, we require that $\widetilde{\mu}$ has finite second moment which is equivalent to the assumption that $\mu$ has finite third moment.
\end{proof}

\section{Proof of Wiener's Test.}
We first divide $\{x\in\R^d: 1\leq |x|<2\}$ into a finite number of small pieces with diameter less than $1/32$: $B_1,\dots,B_N$. Let $K_n^k=K_n\cap(2^nB_k)$ for any $n\in\N^+,1\leq k\leq N$. For any nonempty set $K_n^k$, we have $\diam(K_n^k)\leq2^n/32$ and $\dist(0,K_n^k)\in [2^n,2^{n+1})$. Let $V_n^k$ be the event that $\Snake_0^\infty$ visits $K_n^k$. Applying Theorem \ref{bd-infinite}, we can get:
\begin{equation}
P(V_n^k)\asymp \frac{\BCap(K_n^k)}{2^{n(d-4)}}.
\end{equation}
Since each $K_n^k$ is finite (any finite set is B-transient), we have
$$
P(\Snake_0^\infty \text{ visits } K \;i.o.)=P(V_n^k \;i.o.).
$$

When $\sum_{n=0}^{\infty}\BCap(K_n)/2^{n(d-4)}<\infty$, by monotonicity, for any $1\leq k\leq N$,
$$
\sum_{n=1}^{\infty}\frac{\BCap(K^k_n)}{2^{n(d-4)}}<\sum_{n=0}^{\infty}\frac{\BCap(K_n)}{2^{n(d-4)}}<\infty.
$$
Hence,
$$
\sum_{n=1}^{\infty}\sum_{k=1}^{N}P(V_n^k)\asymp \sum_{n=1}^{\infty}\sum_{k=1}^{N}\frac{\BCap(K^k_n)}{2^{n(d-4)}}
=\sum_{k=1}^{N}\left(\sum_{n=1}^{\infty}\frac{\BCap(K^k_n)}{2^{n(d-4)}}\right)<\infty.
$$
Then by Borel-Cantelli Lemma, almost surely, only finite $V_n^k$ occurs and hence $K$ is B-transient.

When $\sum_{n=0}^{\infty}\BCap(K_n)/2^{n(d-4)}=\infty$, by subadditivity of branching capacity (this can be seen by the limit before \eqref{bd-pp}), we have:
$$
\sum_{n=1}^{\infty}\sum_{k=1}^{N}\frac{\BCap(K^k_n)}{2^{n(d-4)}}
\geq\sum_{n=1}^{\infty}\frac{\BCap(K_n)}{2^{n(d-4)}}=\infty.
$$
Hence for some $1\leq k\leq N$, $\sum_{n=1}^{\infty}\BCap(K^k_n)/2^{n(d-4)}=\infty$. Assume
$$
\sum_{n=1}^{\infty}\BCap(K^1_n)/2^{n(d-4)}=\infty.
$$

We need the following Lemma whose proof we postpone.
\begin{lemma}\label{ineq-two}
There exists some $C>0$, such that, for any $n<m$, we have:
\begin{equation*}
P(V_n^1\cap V_m^1)\leq C P(V_n^1)P(V_m^1).
\end{equation*}
\end{lemma}

Let $I_n=\sum_{i=1}^n \mathbf{1}_{V_i^1}$ and $F=\mathbf{1}_{\{I_n\geq  \E(I_n)/2\}}$. By the lemma above, we have:
$$
\E(I_n^2)\leq C(\E(I_n))^2.
$$
$$
\E (F I_n)=\E I_n-\E(I_n\mathbf{1}_{\{I_n<\E(I_n)/2\}}) \geq \E(I_n)/2.
$$
Hence,
\begin{multline*}
P(I_n\geq\E(I_n)/2)=\E(F)=\E(F^2)\\
\geq(\E (F I_n))^2/\E(I_n^2)
\geq (\E(I_n)/2)^2/C(\E(I_n))^2=1/(4C).
\end{multline*}
Since $\E I_n\rightarrow \infty$, let $n\rightarrow \infty$, we get
$$
P(I_n=\infty)\geq 1/(4C).
$$
By Proposition \ref{recurrent}, we get that $K$ is B-recurrent.

\section{Proof of Lemma \ref{ineq-two}.}
Write $A=K_n^1,B=K_m^1$ and $M=2^m$. Without loss of generality, assume $A,B\neq \emptyset$. We know
$$
\diam(A)\leq 2^n/32, \diam(B)\leq 2^m/32.
$$
Fix any $a\in A$ and $b\in B$. Let $\widehat{A}=a+\Ball(2^n/8)$ and $\widehat{B}=b+\Ball(2^m/8)$. Then we have
\begin{equation}\label{cc}
\dist(A,\widehat{A}^c)\asymp\dist(0,\widehat{A})\asymp 2^n; \dist(B,\widehat{B}^c)\asymp\dist(0,\widehat{B})\asymp 2^m; \dist(a,b)\asymp\dist(\widehat{A},\widehat{B})\asymp 2^m.
\end{equation}
We need to show:
\begin{equation}
P(\Snake_0^\infty \text{ visits both }A\, \&\, B)\preceq \qS_A(0)\qS_B(0).
\end{equation}
In the proof, we will repeatedly use \eqref{bd-pp}, Theorem \ref{bd-infinite}, \eqref{q=cvv-p} and Lemma \ref{CVV} without mention.
Since (see \eqref{two2} and \eqref{adjoint})
\begin{multline*}
P(\Snake_0^\infty \text{ visits both }A\, \&\, B)\preceq\\
\sum_{z\in\Z^d}G_{A\cup B}(0,z)\cdot
\left(\pS_A(z)\qS_B(z)+\pS_B(z)\qS_A(z)+P(\Snake_z\text{ visits both }A\, \&\, B)\right),
\end{multline*}
it suffices to show:
\begin{equation}\label{c1}
\sum_{z\in\Z^d}G_{A\cup B}(0,z)\pS_A(z)\qS_B(z)\preceq \qS_A(0)\qS_B(0);
\end{equation}
\begin{equation}\label{c2}
\sum_{z\in\Z^d}G_{A\cup B}(0,z)\pS_B(z)\qS_A(z)\preceq \qS_A(0)\qS_B(0);
\end{equation}
\begin{equation}\label{c3}
\sum_{z\in\Z^d}G_{A\cup B}(0,z)P(\Snake_z\text{ visits both }A\, \&\, B)\preceq \qS_A(0)\qS_B(0).
\end{equation}
Note that by monotonicity, $G_{A\cup B}(x,y)\leq \min\{G_A(x,y),G_B(x,y)\}$.
For \eqref{c1}, we have:
\begin{align*}
\sum_{z\in\Z^d}&G_{A\cup B}(0,z)\pS_A(z)\qS_B(z)\\
=&\sum_{z\in\widehat{B}}G_{A\cup B}(0,z)\pS_A(z)\qS_B(z)+
\sum_{z\in \widehat{B}^c}G_{A\cup B}(0,z)\pS_A(z)\qS_B(z)\\
\preceq &\sum_{z\in\widehat{B}}G_{B}(0,z)\pS_A(b)\qS_B(z)+
\sum_{z\in \widehat{B}^c}G_{A}(0,z)\pS_A(z)\qS_B(0)\\
\preceq &(\diam (\widehat{B}))^2\qS_B(0)\pS_A(b)+\qS_A(0)\qS_B(0)\preceq\qS_A(0)\qS_B(0).\\
\end{align*}
Similarly one can show \eqref{c2}.

We just need to show \eqref{c3}. We first show that:
\begin{equation}\label{pAB}
P(\Snake_z\text{ visits both }A\, \&\, B)\preceq \left\{
\begin{array}{ll}
\pS_A(b)\qS_B(z)+\pS_B(0)\qS_A(z);& \text{when }z\in \Ball(4M);\\
\pS_A(z)\qS_B(a);& \text{when }z\notin \Ball(4M).\\
\end{array}\right.
\end{equation}

By \eqref{two1}, we need to estimate:
$$
\sum_{w\in\Z^d}G_{A\cup B}(z,w)\pS_A(w)\pS_B(w).
$$
When $z\in\Ball(4M)$, we have
\begin{align*}
\sum_{w\in\Z^d}&G_{A\cup B}(z,w)\pS_A(w)\pS_B(w)\\
&=\sum_{w\in\widehat{B}}G_{A\cup B}(z,w)\pS_A(w)\pS_B(w)+\sum_{w\in\widehat{B}^c}G_{A\cup B}(z,w)\pS_A(w)\pS_B(w)\\
&\preceq \sum_{w\in\widehat{B}}G_{A}(z,w)\pS_A(b)\pS_B(w)+\sum_{w\in\widehat{B}^c}G_{B}(z,w)\pS_A(w)\pS_B(0)\\
&\preceq \pS_A(b)\qS_B(z)+\pS_B(0)\qS_A(z).\\
\end{align*}
When $z\notin\Ball(4M)$, let $\widehat{C}=\Ball(3M)$. We divide the sum into three parts:
$$
\sum_{w\in \widehat{B}},\;\sum_{w\in \widehat{C}\setminus \widehat{B}},\;\sum_{w\in \widehat{C}^c}.
$$
\begin{align*}
\sum_{w\in \widehat{B}}& G_{A\cup B}(z,w)\pS_A(w)\pS_B(w)
\preceq\sum_{w\in\widehat{B}}G_{B}(z,w)\pS_A(b)\pS_B(w)\\
&\preceq (\diam{\widehat{B}})^2\pS_B(z)\pS_A(b)\asymp(\dist(a,b))^2\frac{\BCap(A)\BCap(B)}{(\dist(z,B))^{d-2}(\dist(a,b)^{d-2})}\\
&\asymp\pS_A(z)\qS_B(a).
\end{align*}
\begin{align*}
\sum_{w\in \widehat{C}\setminus\widehat{B}}& G_{A\cup B}(z,w)\pS_A(w)\pS_B(w)
\preceq\sum_{w\in\widehat{C}\setminus\widehat{B}}G_{A}(z,w)\pS_A(w)\pS_B(a)\\
&\preceq (\diam{\widehat{C}})^2\pS_A(z)\pS_B(a)\asymp\pS_A(z)\qS_B(a).
\end{align*}
\begin{align*}
\sum_{w\in \widehat{C}^c}&G_{A\cup B}(z,w)\pS_A(w)\pS_B(w)
\preceq\sum_{w\in \widehat{C}^c}g(z,w)\frac{\BCap(A)\BCap(B)}{|w-a|^{d-2}|w-b|^{d-2}}\\
\asymp&\BCap(A)\BCap(B)\sum_{w\in \widehat{C}^c}\frac{1}{|w-z|^{d-2}|w-a|^{2d-4}}\\
\stackrel{(\ast)}{\preceq}& \frac{\BCap(A)\BCap(B)}{|z-a|^{d-2}|b-a|^{d-4}}\asymp \pS_A(z)\qS_B(a).\\
\end{align*}
Combine all three above, we get \eqref{pAB}. Note that for $(\ast)$, we use:
\begin{align*}
\sum_{w\in \widehat{C}^c}&\frac{1}{|w-z|^{d-2}|w-a|^{2d-4}}\\
\leq&\sum_{\|w-z\|\leq \|z\|/8}\frac{1}{|w-z|^{d-2}|w-a|^{2d-4}}
+\sum_{\|w-z\|\geq \|z\|/8, w\in\widehat{C}^c}\frac{1}{|w-z|^{d-2}|w-a|^{2d-4}}\\
\preceq&\sum_{\|w-z\|\leq \|z\|/8}\frac{1}{|w-z|^{d-2}|z-a|^{2d-4}}
+\sum_{\|w-z\|\geq \|z\|/8, w\in\widehat{C}^c}\frac{1}{|z|^{d-2}|w-a|^{2d-4}}\\
\preceq&\frac{|z|^2}{|z-a|^{2d-4}}+\frac{1}{|z|^{d-2}}\sum_{w\in\widehat{C}^c}\frac{1}{|w-a|^{2d-4}}
\asymp\frac{1}{|z|^{2d-6}}+\frac{1}{|z|^{d-2}}\sum_{n\geq 3M}\frac{n^{d-1}}{n^{2d-4}}\\
\asymp&\frac{1}{|z|^{2d-6}}+\frac{1}{|z|^{d-2}}\frac{1}{M^{d-4}}
\asymp\frac{1}{|z|^{d-2}}\frac{1}{M^{d-4}}\asymp\frac{1}{|z-a|^{d-2}|b-a|^{d-4}}.\\
\end{align*}

Hence,
\begin{align*}
&\sum_{z\in \Z^d}G_{A\cup B}(0,z)P(\Snake_z\text{ visits both }A\, \&\, B)\\
\preceq& \sum_{z\in \Ball(4M)}G_{A\cup B}(0,z)(\pS_A(b)\qS_B(z)+\pS_B(0)\qS_A(z))+\sum_{z\in \Ball(4M)^c}G_{A}(0,z)\pS_A(z)\qS_B(a)\\
\preceq& M^2\pS_A(b)\qS_B(0)+M^2\pS_B(0)\qS_A(0)+\qS_A(0)\qS_B(a)\asymp\qS_B(0)\qS_A(0).\\
\end{align*}
This is just \eqref{c3} and we finish the proof.

\section*{Acknowledgements}

We thank Professor Omer Angel for valuable comments on earlier versions of this paper.

\end{document}